\documentclass[12pt, a4paper]{amsart}

\usepackage{amsmath,amssymb}
\usepackage{amsthm,amscd}
\usepackage{enumerate}

\newtheorem{theorem}{Theorem}[section]
\newtheorem{corollary}[theorem]{Corollary}
\newtheorem{lemma}[theorem]{Lemma}
\newtheorem{proposition}[theorem]{Proposition}

\newtheorem{remark}[theorem]{Remark}
\newtheorem{example}[theorem]{Example}

\begin{document}

\title[Polynomial skew products]
{Polynomial skew products whose Julia sets have infinitely many symmetries}
\author[Kohei Ueno]{Kohei Ueno}
\address{Daido University, Nagoya 457-8530, Japan}
\curraddr{}
\email{k-ueno@daido-it.ac.jp}
\urladdr{}
\dedicatory{}
\date{}
\thanks{}
\translator{}

\begin{abstract}
We consider the symmetries of Julia sets of
polynomial skew products on $\mathbb{C}^2$, 
which are birationally conjugate to rotational products. 
Our main results give the classification of 
the polynomial skew products 
whose Julia sets have infinitely many symmetries.
\end{abstract}

\maketitle

\section{Introduction}

Any kind of Julia sets of a polynomial map can have symmetries. 
We say that a Julia set has symmetries 
if some nonelementary transformations preserve it.
Beardon~\cite{b} investigated the symmetries of 
the Julia sets of polynomials on $\mathbb{C}$. 
He considered conformal functions as symmetries,
and proved that the group of symmetries is infinite
if and only if the polynomial is conjugate to $z \to z^{\delta}$. 
To generalize the results in one-dimension 
to those in higher dimensions, 
we~\cite{u} previously investigated 
the symmetries of Julia sets of 
nondegenerate polynomial skew products on $\mathbb{C}^2$. 
We defined the Julia sets as the supports of the Green measures, 
which are compact, 
and considered suitable polynomial automorphisms as symmetries. 
In this paper, 
we investigate the symmetries of Julia sets of 
polynomial skew products on $\mathbb{C}^2$.
We define the Julia sets by fiberwise Green functions, 
which are close to the supports of the Green measures.
However, the Julia sets may no longer be compact,
and the study of the dynamics becomes more difficult. 
We obtain similar results to the nondegenerate case,
with some assumptions in several statements,
and succeed in classifying the polynomial skew products 
whose Julia sets have infinitely many symmetries.

A polynomial skew product on $\mathbb{C}^2$ is 
a polynomial map of the form $f(z,w) = (p(z),q(z,w))$.
More precisely, 
let $p(z) = a_{\delta} z^{\delta} + O(z^{\delta - 1})$
and $q(z,w) = b_{d}(z) w^{d} + O_z (w^{d-1})$,
where $\delta \geq 2$ and $d \geq 2$.

Our paper is organized as follows.
In Section 2, 
we briefly recall the dynamics of polynomials
and the relevant results on the symmetries 
of Julia sets. 
In Section 3,
we recall the dynamics of polynomial skew products.
In particular,
we review the existence of 
the fiberwise Green and B\"{o}ttcher functions,
and give the definition of Julia sets. 
The study of the symmetries of Julia sets 
begins in Section 4.
First,
we define the centroids of $f$ as defined in \cite{b} and \cite{u},
and show that the symmetries of the Julia set of $f$ are 
birationally conjugate to rotational products.
The tools for the proof are 
the fiberwise Green and B\"{o}ttcher functions of $f$,
which also relate to the centroids of $f$.
Next, 
under some assumptions,
we characterize the group of symmetries 
by functional equations including the iterates of $f$.
The assumptions are, for example,
the normality of $f$ and 
the special form of the polynomial $b_d$.
The normality of $f$
means that the centroids are at the origin.
Finally,
we classify the polynomial skew products 
whose Julia sets have infinitely many symmetries 
in Section 5.
We have two main theorems based on two cases: 
that when the map is in normal form and the general case.  
We find that these maps can be classified into four types. 

\section{Symmetries of Julia sets of polynomials}

In this section,
we recall the dynamics of polynomials on $\mathbb{C}$ and 
the relevant results on the symmetries of the Julia sets of polynomials.

Let 
$p(z) = a_{\delta} z^{\delta} + a_{\delta-1} z^{\delta-1} + \cdots + a_0$
be a polynomial of degree $\delta \geq 2$.
We denote by $p_2 p_1$ the composition of polynomials $p_1$ and $p_2$: 
$p_2 p_1 (z) = p_2 (p_1 (z))$.
Let $p^n$ be the $n$-th iterate of $p$.
A useful tool for the study of the dynamics of $p$ is 
the Green function $G_p$ of $p$,
\[
G_p(z) = \lim_{n \to \infty} {\delta}^{-n} \log^+ |p^n(z)|.
\]
It is well known that $G_p$ is a nonnegative, 
continuous and subharmonic function on $\mathbb{C}$.
By definition, $G_p(p(z)) = \delta G_p(z)$.
Moreover,
$G_p$ is harmonic on $\mathbb{C} \setminus K_p$ and zero on $K_p$, 
where $K_p = \{ z : \{ p^n(z) \}_{n \geq 1} \text{ bounded} \}$, and 
$G_p(z) = \log |z| + (\delta - 1)^{-1} \log |a_{\delta}| + o(1)$ 
as $z \to \infty$.
This is the Green function for $K_p$ with a pole at infinity,
determined only by the the compact set $K_p$.
This function induces the B\"{o}ttcher function $\varphi_p$ 
defined near infinity such that 
$\varphi_p (z) = z + O (1) $  as  $z \to \infty$, 
$\log |c \varphi_p (z)| = G_p(z)$, 
where $c = \sqrt[\delta - 1]{a_{\delta}}$, and
$\varphi_{p} ( p(z) ) = a_{\delta} ( \varphi_p (z) )^{\delta}$.

Let us recall some objects and results of 
the symmetries of the Julia sets of polynomials.  
For further details, see \cite{b}. 
We define the Julia set $J_p$ of $p$ as the boundary $\partial K_p$,
and consider conformal functions 
as the symmetries of $J_p$.
Since $J_p$ is compact,
such functions are Euclidean isometries. 
Thus the group of the symmetries of $J_p$ is defined by
\[
\Sigma_p
= \{ \sigma \in E : \sigma (J_p) = J_p \},
\]
where 
$E = \{ \sigma (z) = c_1 z + c_2 : |c_1|=1, c_1, c_2 \in \mathbb{C} \}$.
The centroid of $p$ is defined by
\[
\zeta = \frac{-a_{\delta -1}}{\delta a_{\delta}}.
\]
Each symmetry $\sigma$ is a rotation about the centroid: 

\begin{proposition}[{\cite[Theorem 5]{b}}]\label{poly, centroid} 
For any symmetry $\sigma$ in $\Sigma_p$, 
there exists $\mu$ in the unit circle $S^1$ such that
$\sigma (z)= \mu (z- \zeta ) + \zeta$.
\end{proposition}

We can characterize $\Sigma_p$ by the unique equation.

\begin{proposition}[{\cite[Lemma 7]{b}}]\label{poly, fun-eq} 
It follows that
$\Sigma_p = \{ \sigma \in E : p \sigma = \sigma^{\delta} p \}$.
\end{proposition}

By Proposition {\rmfamily \ref{poly, centroid}},
the group $\Sigma_p$ is identified with 
a subgroup of the unit circle $S^1$.
This group is trivial, finite cyclic or infinite. 
We have a sufficient and necessary condition 
for $\Sigma_p$ to be infinite.

\begin{proposition}[{\cite[Lemma 4]{b}}]\label{poly, infinite} 
The group $\Sigma_p$ is infinite if and only if
$p$ is affinely conjugate to $z^{\delta}$,
or equivalently, if $J_p$ is a circle.
In this case, 
$\Sigma_p$ consists of all rotations about $\zeta$.
\end{proposition}

We say that $p$ is in normal form 
if $a_{\delta} = 1$ and $a_{\delta - 1} = 0$,
so that the centroid is at the origin.
We may assume that $p$ is in normal form 
without loss of generality
because $p$ is conjugate to a polynomial in normal form. 
With this terminology,
we can restate Proposition {\rmfamily \ref{poly, infinite}} 
as follows.

\begin{proposition} 
\label{poly-normal, infinite} 
Let $p$ be in normal form.
Then, $\Sigma_p$ is infinite if and only if $p(z) = z^{\delta}$,
or equivalently, if $J_p = S^1$. 
In this case, $\Sigma_p \simeq S^1$. 
\end{proposition}

We can completely determine the group $\Sigma_p$
even if it is finite.

\begin{proposition}[{\cite[Theorem 5]{b}}]\label{poly, order} 
Let $p$ be in normal form.
If $\Sigma_p$ is finite, 
then the order of $\Sigma_p$ is equal to 
the largest integer $m$ such that 
$p$ can be written in the form $p(z) = z^r Q(z^m)$
for some polynomial $Q$.
\end{proposition}

The tools for the proofs of these facts 
are the Green and B\"{o}ttcher functions of $p$.
We generalize 
Propositions {\rmfamily \ref{poly, centroid}} 
and {\rmfamily \ref{poly, fun-eq}}
in Section 4,
and Propositions {\rmfamily \ref{poly, infinite}}
and {\rmfamily \ref{poly-normal, infinite}}
in Section 5.
We use Proposition {\rmfamily \ref{poly, order}}
to prove the main theorems in Section 5.

\section{Dynamics of polynomial skew products}

In this section, 
we recall the dynamics of polynomial skew products on $\mathbb{C}^2$
and give the definition of Julia sets. 

\subsection{Polynomial skew products}

A polynomial skew product on $\mathbb{C}^2$ is a polynomial map of the form 
$f(z,w) = (p(z),q(z,w))$.
Let
\[ 
\left\{
   \begin{array}{l}
   p(z) = a_{\delta} z^{\delta} + a_{\delta-1} z^{\delta-1} + \cdots + a_0, \\
   q(z,w) = b_{d}(z) w^{d} + b_{d-1}(z) w^{d-1} + \cdot \cdot + b_{0}(z),
   \end{array}
\right.
\]
and let $b_d$ be a polynomial in $z$ of degree $l \geq 0$.
We assume that $\delta \geq 2$ and $d \geq 2$.
As in \cite{u},
we say that $f$ is nondegenerate 
if $b_d$ is a nonzero constant.

Let us briefly recall 
the dynamics of polynomial skew products.
Roughly speaking, 
the dynamics of $f$ consists of 
the dynamics on the base space and on the fibers.
The first component $p$ defines 
the dynamics on the base space $\mathbb{C}$.
Note that $f$ preserves the set of vertical lines in $\mathbb{C}^{2}$.
In this sense, 
we often use the notation $q_z (w)$ instead of $q(z,w)$.
The restriction of $f^n$ to the vertical line $\{ z \} \times \mathbb{C}$
is viewed as the composition of $n$ polynomials on $\mathbb{C}$, 
$q_{p^{n-1}(z)} \cdots q_{p(z)} q_{z}$.
Therefore,
the $n$-th iterate of $f$ is written as follows: 
\[
f^n(z,w) = (p^n(z), Q_z^n(w)),
\]
where $Q_z^n(w) = q_{p^{n-1}(z)} \cdots q_{p(z)} q_z (w)$.

\subsection{Fiberwise Green and B\"{o}ttcher functions} 

We saw that 
the Green function of polynomial $p$ is well defined.
In a similar fashion, 
we define the fiberwise Green function $G_z$ of $f$ as follows:
\[
G_z(w) = \lim_{n \to \infty} d^{-n} \log^+ |Q_z^n(w)|.
\]
From the result of Favre and Guedj in \cite[Theorem 6.1]{fg}, 
it follows that $G_z$ defines 
a local bounded function on $K_p \times \mathbb{C}$
such that  
$G_{p(z)} (q_z(w)) = d G_z(w)$.
However, it is not continuous in general;
if $b_d^{-1} (0) \cap K_p = \emptyset$, 
then it is continuous on $K_p \times \mathbb{C}$.
We define $K_z$ as $\{ w : G_z(w) = 0 \}$,
which is nonempty for any $z$ in $K_p$. 
To describe $G_z$ more precisely, we define 
\[
\Phi (z) = \sum_{n = 0}^{\infty} \dfrac{1}{d^{n+1}} \log |b_d(p^n(z))|.
\]
It belongs to $L^1 (\mu_p)$, 
where $\mu_p$ is the Green measure of $p$.
For fixed $z$ in $K_p \setminus \{ \Phi = - \infty \}$,
the function $G_z$ is nonnegative, 
continuous and subharmonic on $\mathbb{C}$.
More precisely, 
it is harmonic on $\mathbb{C} \setminus K_z$ and 
$G_z(w) = \log |w| + \Phi (z) + o_z(1)$ as $w \to \infty$.
This is the Green function 
for the compact set $K_z$ with a pole at infinity.
Note that $K_p \setminus \{ \Phi = - \infty \}$ is 
forward invariant under $p$. 

The fiberwise Green function $G_z$ induces 
the fiberwise B\"{o}ttcher function $\varphi_z$,
which is useful to investigate the symmetries of Julia sets.

\begin{lemma}\label{Bottcher} 
For every $z$ in $K_p \setminus \{ \Phi = - \infty \}$,
there exists a unique conformal function $\varphi_z$ 
defined near infinity such that
\begin{enumerate} [(i)]
\item ${\varphi}_z (w) = w + O_z (1) $  as  $w \to \infty$, 
\item $\log |c_z \varphi_z (w)| = G_z (w)$, where $c_z = \exp (\Phi (z))$, 
\item $\varphi_{p(z)} ( q_z (w) ) = b_d (z) ( {\varphi_z} (w) )^d$.
\end{enumerate}
\end{lemma}

\subsection{Julia sets}

In this paper, we consider the following Julia set:
\[
J_f = \bigcup_{z \in J_p} \{ z \} \times \partial K_z.
\]
Here we define $\partial K_z$ as $\emptyset$
if $K_z = \mathbb{C}$.
It follows that $J_f$ is forward invariant under $f$; 
that is, $f(J_f) \subset J_f$.
If $b_d^{-1} (0) \cap J_p = \emptyset$, 
then $J_f$ is completely invariant under $f$. 
Moreover, $J_f$ is compact 
if and only if $b_d^{-1} (0) \cap J_p = \emptyset$.

The following subset of $J_p$ plays an important role in the proofs:
\[
J_p^* = J_p \setminus \{ \Phi =  - \infty \}.
\]
Note that $J_p^*$ is dense in $J_p$ 
because it contains most periodic points.
Moreover, $\mu_p (J_p^*) = 1$.
For any $z$ in $J_p^*$,
the limits $G_z$ and $\varphi_z$ are well defined.
In addition, $J_p^*$ is forward invariant under $p$, and
$J_p^* \setminus p(J_p^*) \subset p(b_d^{-1} (0))$.

There is another Julia set of $f$ 
that might be appropriately called the Julia set of $f$.
Favre and Guedj proved in \cite[Theorem 6.3]{fg} that the closure
\[
\overline{ \bigcup_{z \in J_p^*} \{ z \} \times \partial K_z} 
\]
coincides with the support of the Green measure of $f$.
Similar to $J_f$, this Julia set is compact 
if and only if $b_d^{-1} (0) \cap J_p = \emptyset$.

\begin{remark}
Almost the same claims as in the next two sections
hold for the symmetries of the last Julia set
under the assumption $b_d^{-1} (0) \cap J_p = \emptyset$.
\end{remark}

\section{Symmetries of Julia sets} 

As a symmetry, we consider a polynomial automorphism
of the form $\gamma (z,w) = (\gamma_1 (z), \gamma_2 (z,w))$
that preserves $J_f$.
Since $\gamma_1$ is conformal, $\gamma_1 (z) = c_1 z + c_2$,
where $c_1$ and $c_2$ are complex numbers.
Since $J_p$ is compact, $|c_1| = 1$.
Since $\gamma_2$ is conformal on each fiber, 
$\gamma_2 (z, w) = c_3 w + c_4 (z)$,
where $c_3$ is a complex number and $c_4$ is a polynomial in $z$.
Since $K_z$ is compact for some $z$ in $J_p$, 
it follows that $|c_3| = 1$.
Therefore, we define the group of the symmetries of $J_f$ as
\[
\Gamma_f 
=
\{ \gamma \in S : \gamma ( J_f ) = J_f \},
\]
where
\[
S 
=
\left\{ \
\gamma 
\left(
\begin{array}{ccccc}
z \\
w 
\end{array}
\right)
=
\left(
\begin{array}{ccccc}
c_1 z + c_2 \\
c_3 w + c_4 (z) 
\end{array}
\right)
:
|c_1| = |c_3| = 1
\right\}.
\]
Let us denote $\gamma$ in $\Gamma_f$ by $(\sigma (z), \gamma_{z} (w))$.
Since $\sigma$ preserves $J_p$, 
it follows that $\sigma$ belongs to $\Sigma_p$.
By definition, 
$\gamma_z (\partial K_z) = \partial K_{\sigma (z)}$ and so 
$\gamma_z (K_z) = K_{\sigma (z)}$ for any $z$ in $J_p$.

\subsection{Centroids} 

As defined in Section 2, 
we define the centroids of $f$ as
\[
\zeta = \frac{- a_{\delta-1}}{\delta a_{\delta}} \ 
\text{ and } \ 
\zeta_z = \frac{- b_{d-1}(z)}{d b_d (z)}.
\]
Although $\zeta$ is a constant, 
$\zeta_z$ is a rational function in $z$.

The fiberwise B\"{o}ttcher function $\varphi_z$
relates to the centroid ${\zeta}_{z}$. 
The following lemma follows from $(i)$ and $(iii)$ 
in Lemma {\rmfamily \ref{Bottcher}}.

\begin{lemma} \label{Bottcher & centroid} 
It follows that
$\varphi_z (w) = w - {\zeta}_{z} + o_z (1)$
for any $z$ in $J_p^*$.
\end{lemma}

We first show that 
a symmetry $\gamma$ is birationally conjugate to a rotational product.

\begin{proposition}\label{centroid} 
For any $\gamma$ in $\Gamma_f$, 
there exist $\mu$ and $\nu$ in $S^1$ such that
\[
\gamma
\left(
\begin{array}{ccccc}
z \\
w 
\end{array}
\right)
=
\left(
\begin{array}{ccccc}
\mu ( z - \zeta ) + \zeta \\
\nu ( w - {\zeta}_{z} ) + {\zeta}_{\sigma (z)} 
\end{array}
\right),
\]
where $\sigma (z) = \mu ( z - \zeta ) + \zeta$ belongs to $\Sigma_p$.
\end{proposition}

\begin{proof}
Because of Lemma {\rmfamily \ref{Bottcher}}
and the definition of the symmetry,
the equation $G_{\sigma (z)} \gamma_z = G_z$ implies 
the equation $\varphi_{\sigma (z)} \gamma_z = \nu \varphi_z$
for some $\nu$ in $S^1$
over $J_p^* \cap \sigma^{-1} J_p^*$. 
Hence, 
it follows from Lemma {\rmfamily \ref{Bottcher & centroid}}
that $\gamma_z (w) = \nu ( w - {\zeta}_{z} ) + {\zeta}_{\sigma (z)}$
over $J_p^* \cap \sigma^{-1} J_p^*$, which is dense in $J_p$.
By the identity theorem of holomorphic functions on horizontal lines,
the equation extends to $\{ b_d(z) b_d (\sigma(z)) \neq 0 \} \times \mathbb{C}$.
By Riemann's removable singularity theorem,
$\nu ( w - {\zeta}_{z} ) + {\zeta}_{\sigma (z)}$ 
is also polynomial on $\mathbb{C}^2$.
Therefore, the equation holds on $\mathbb{C}^2$.
\end{proof}

\begin{corollary} 
It follows that $\sigma$ completely preserves 
the set $\{ z \in J_p : \zeta_z = \infty \}$,
where $\sigma$ is the first component of $\gamma$ in $\Gamma_f$. 
\end{corollary}

By Proposition {\rmfamily \ref{centroid}},
we can identify $\Gamma_f$ with a subgroup of the torus:
\[
\Gamma_f 
\simeq \{ (\mu, \nu) \in S^1 \times S^1: \gamma_{\mu, \nu} \in \Gamma_f \}
\subset S^1 \times S^1,
\]
where
\[
\gamma_{\mu, \nu}
\left(
\begin{array}{ccccc}
z \\
w 
\end{array}
\right)
=
\left(
\begin{array}{ccccc}
\mu ( z - \zeta ) + \zeta \\
\nu ( w - {\zeta}_{z} ) + {\zeta}_{\sigma (z)} 
\end{array}
\right).
\]
We use the notation $=$ instead of $\simeq$
hereafter. 
By definition,
$\Gamma_f \subset \Sigma_p \times S^1$.

\subsection{Normal form}

Referring to the polynomial case,
we say that $f$ is in normal form
if $p$ and $b_d$ are monic
and $a_{\delta - 1}$ and $b_{d - 1}$ are the constant $0$.
Roughly speaking,
the normality of $f$ is equivalent to those of $p$ and $q_z$.
Hence if $f$ is in normal form,
then the centroids are at the origin.

Unlike the cases of polynomials
and nondegenerate polynomial skew products,
we may not assume that $f$ is in normal form
without loss of generality.
However, 
we can normalize $f$ to a rational map  
as follows.
Define $h(z,w) = (c_1 (z - \zeta), c_2 (w - \zeta_z))$,
where $c_1^{\delta - 1}$ is equal to $a_{\delta}$,
the coefficient of the leading term of $p$,
and $c_1^l c_2^{d - 1}$ is equal to 
the coefficient of the leading term of $b_d$.
Then $h$ is a birational map.
Let $\tilde{f}$ be the conjugation of $f$ by $h$: 
$h f = \tilde{f} h$.
This rational map satisfies 
all conditions in the definition of normality.
Hence we call $\tilde{f}$ 
the normalized rational skew product of $f$,
which appears in Section 5.2.

\subsection{Functional equations}

Under some assumptions,
we characterize $\Gamma_f$ by functional equations
including the iterates of $f$.
Although the group $\Sigma_p$ of a polynomial $p$
is characterized by the unique equation
$p \sigma = \sigma^{\delta} p$,
our characterization of $\Gamma_f$ needs
infinitely many equations
as in \cite[Lemma 3.2]{u}.
Referring to Proposition {\rmfamily \ref{centroid}},
we define 
\[
\mathcal{E}_f = \{ \gamma \in S:
f^n \gamma = \gamma_n f^n \text{ for } \forall n \geq 1 \},
\]
where
\[
\gamma_n
\left(
\begin{array}{ccccc}
z \\
w 
\end{array}
\right)
=
\left(
\begin{array}{ccccc}
\mu^{\delta^n} (z - \zeta) + \zeta \\
\mu^{l_n} \nu^{d^n} (w - {\zeta}_{p^n(z)}) + {\zeta}_{p^n(\sigma (z))} 
\end{array}
\right)
\]
and $l_n = l(\delta^n - d^n)/(\delta - d)$.
Unlike the nondegenerate case,
we need some assumptions for $\Gamma_f$
to coincide with $\mathcal{E}_f$, 
which may be removable.

We first provide a lemma 
about certain symmetries of the polynomial $b_d$.

\begin{lemma}\label{sym of b_d} 
The identity $|b_d (\sigma (z))| = |b_d (z)|$ holds
for any symmetry $\sigma$ and  
for any $z$ in $J_p^* \cap \sigma^{-1} J_p^*$,
where $\sigma$ is the first component of $\gamma$ in $\Gamma_f$.
\end{lemma}

\begin{proof}
Let $K_z^c = \mathbb{C} \setminus K_z$.
For any $z$ in $J_p^* \cap \sigma^{-1} J_p^*$,
the following commutative diagram holds:
\begin{equation*}
\begin{CD}
 K_{p(z)}^c @<\text{$q_z$}<< K_z^c @>\text{$\gamma_z$}>> K_{\sigma (z)}^c 
 @>\text{$q_{\sigma (z)}$}>> K_{p \sigma (z)}^c \\
 @VV\text{$\varphi_{p(z)}$}V @V\text{$\varphi_{z}$}VV 
 @V\text{$\varphi_{\sigma (z)}$}VV @V\text{$\varphi_{p \sigma (z)}$}VV \\
 \mathbb{C} @<\text{$b_d(z) w^d$}<< \mathbb{C} @>\text{$\nu w$}>> 
 \mathbb{C} @>\text{$b_d (\sigma (z)) w^d$}>> \mathbb{C}.
\end{CD}
\vspace{4mm}
\end{equation*}
Although $\varphi_z$ may not be defined on $K_z^c$, 
it is defined near infinity.
Since $p \sigma = \sigma^{\delta} p$, 
the Euclidean isometry 
$\gamma_{\sigma^{\delta - 1} p(z)} \cdots \gamma_{\sigma p(z)} \gamma_{p(z)}$
maps $K_{p(z)}$ onto $K_{p \sigma (z)}$.
Therefore, $|b_d (\sigma (z)) (\nu w)^d| = |b_d (z) w^d|$.
\end{proof}

We use this lemma to prove the main theorems in the next section.
It is natural to ask whether or not 
the equation $b_d (\sigma (z)) = \mu^l b_d (z)$ holds.
In the following proposition,
we assume some conditions that guarantee this equation.

\begin{proposition} \label{fun-eq, b = z^l} 
If $p$ is in normal form and $b_d (z) = z^l$, then
$\Gamma_f \subset \mathcal{E}_f$.
Moreover, $\sigma (J_p^*) = J_p^*$,
where $\sigma$ is the first component of $\gamma$ in $\Gamma_f$.
\end{proposition}

\begin{proof} 
Let $\gamma$ be a symmetry in $\Gamma_f$.
Proposition {\rmfamily \ref{poly, fun-eq}} 
implies that $p \sigma = \sigma^{\delta} p$
and so $p^n \sigma = \sigma^{\delta^n} p^n$; that is, 
$p^n (\sigma (z)) = \mu^{\delta^n} (p^n (z) - \zeta) + \zeta$.

We first show that $f \gamma = \gamma_1 f$.
It follows from the commutative diagram above that
$\varphi_{p(z)} (q_z (w)) = b_d (z) (\varphi_z (w))^d$
and $\varphi_{p \sigma (z)} (q_{\sigma (z)} \gamma_z (w))
= b_d (\sigma (z)) (\nu \varphi_z (w))^d$
over $J_p^* \cap \sigma^{-1} J_p^* \cap \sigma^{-1} p^{-1} J_p^*$.
Thus
$b_d (z) \varphi_{p \sigma (z)} (q_{\sigma (z)} $ $\gamma_z (w)) 
= b_d (\sigma (z)) \nu^d \varphi_{p(z)} (q_z (w))$,
which implies that
\[
b_d (z) (q_{\sigma (z)} (\gamma_z (w)) - \zeta_{p \sigma (z)}) 
= b_d (\sigma (z)) \nu^d (q_z (w) - \zeta_{p(z)}).
\]
As in the proof of Proposition {\rmfamily \ref{centroid}},
we can extend this equation to $\mathbb{C}^2$.
By assumption, 
\[
b_d (\sigma (z)) = \mu^l  b_d (z).
\]
Therefore,
$q_{\sigma (z)} (\gamma_z (w)) 
= \mu^l \nu^d (q_z (w) - \zeta_{p(z)}) + \zeta_{p \sigma (z)}$.  

Next, we show that $f^n \gamma = \gamma_n f^n$ for any $n \geq 2$.
Although we do not know whether $\gamma_1$ belongs to $\Gamma_f$,
similar arguments as above induce the equation
$Q_{\sigma (z)}^2 (\gamma_z (w))$ $= \mu^{(\delta + d)l} \nu^{d^2}
(Q_z^2 (w) - \zeta_{p^2 (z)}) + \zeta_{p^2 \sigma (z)}$
over $J_p^* \cap ( \cap_{j = 0}^{2} \sigma^{-1} p^{-j} J_p^* )$.
Similarly, we have the equation
$Q_{\sigma (z)}^n (\gamma_z (w)) = 
\mu^{l_n} \nu^{d^n} (Q_z^n (w) - \zeta_{p^n(z)}) + \zeta_{p^n \sigma (z)}$
over $J_p^* \cap ( \cap_{j = 0}^{n} \sigma^{-1} p^{-j} J_p^* )$,
which extends to $\mathbb{C}^2$.

Finally, 
the equation $b_d \sigma = \mu^l b_d$ implies 
$b_d (p^n \sigma (z)) = \mu^{l \delta^n} b_d (p^n (z))$.
Thus $\Phi \sigma = \Phi$
and so $\sigma (J_p^*) = J_p^*$.
\end{proof}

With a slight change in the proof,
we can replace the assumption in 
this proposition with the assumption that 
$q$ is not divisible by any nonconstant polynomial in $z$.

We use the following corollary of 
Proposition {\rmfamily \ref{fun-eq, b = z^l}}
to calculate the groups of symmetries for some examples
in Section 4.4
and to prove the main theorems in Section 5. 

\begin{corollary}\label{cor of fun-eq, b = z^l} 
If $f$ is in normal form and $b_d (z) = z^l$,
then 
\[
q(\mu z, \nu w) = \mu^{l} \nu^{d} q(z,w)
\]
for any $\gamma (z, w) = (\mu z, \nu w)$ in $\Gamma_f$.
\end{corollary}

For the inverse inclusion, 
we have the following. 

\begin{proposition}\label{inverse}
If $b_d^{-1} (0) \cap J_p = \emptyset$ or $b_{d-1} (z) \equiv 0$, 
then $\Gamma_f \supset \mathcal{E}_f$.
\end{proposition}

\begin{proof} 
For both cases,
the equation $f^n \gamma = \gamma_n f^n$ induces
that $G_{\sigma (z)} \gamma_z = G_z$ 
and so $\gamma_z (K_z) = K_{\sigma (z)}$
for any $z$ in $J_p$.
Hence
$\gamma$ in $\mathcal{E}_f$ preserves $J_f$.
\end{proof}

Combining Propositions {\rmfamily \ref{fun-eq, b = z^l}}
and {\rmfamily \ref{inverse}},
we get sufficient conditions for $\Gamma_f$
to coincide with $\mathcal{E}_f$.

\begin{corollary}\label{Gamma = F} 
Assume that $f$ satisfies one of the following conditions$:$
$(i)$ $f$ is in normal form and $q$ is not divisible by 
any nonconstant polynomial in $z$, 
$(ii)$ $f$ is in normal form and $b_d (z) = z^l$,
$(iii)$ $p(z) = z^{\delta}$ and $b_d (z) = z^l$.
Then $\Gamma_f = \mathcal{E}_f$.
In particular, $\Gamma_f$ is compact,
and $\gamma_n$ belongs to $\Gamma_f$ for any $n \geq 1$
if $\gamma$ belongs to $\Gamma_f$.
\end{corollary}

\subsection{Examples}

Let us provide some examples 
of the groups of symmetries. 
For a map $f$ of the four examples below,
the symmetries have to satisfy the equation in
Corollary {\rmfamily \ref{cor of fun-eq, b = z^l}}.
Moreover,
by Corollary {\rmfamily \ref{Gamma = F}},
the pair of two numbers 
satisfying the equation in
Corollary {\rmfamily \ref{cor of fun-eq, b = z^l}}
belongs to $\Gamma_f$ if and only if
it satisfies the infinitely many equations 
in $\mathcal{E}_f$.

\begin{example}
Let $f(z,w) = (z^3, z w^2 + z)$.
Then $\Gamma_f = \{ ( \mu ,\nu ) : \mu^2 = \nu^2 = 1 \} 
= \{ (1,1), (-1$, $-1), (1,-1), (-1,1) \}$.
Moreover,
let $g(z,w) = (z^3, zw^2 + 2 z^2 w + z)$.
Then it is conjugate to $f$ by $h(z,w)=(z, w - z):$ 
$hf = gh$.
Hence $\Gamma_g = \{ (z,w), (-z,-w), (z,-w-2z), (-z,w+2z) \}$.
\end{example}

\begin{example}\label{ex2} 
Let $f(z,w) = (z^2 - 1, z^2 w^2)$.
Then $\Gamma_f = \{ \pm 1 \} \times S^1$.
\end{example}

\begin{example}\label{ex3} 
Let $f(z,w) = (z^3, z w^2 + z^3)$.
Then $\Gamma_f = \{ ( \mu ,\nu ) : \mu^2 = \nu^2 \in S^1 \}$.
Moreover, $f$ is semiconjugate to $f_0 (z,w) = (z^3, w^2 + 1)$
by $\pi (z, w) = (z, zw):$  
$\pi f_0 = f \pi$.
\end{example}

\begin{example}\label{ex4} 
Let $f(z,w) = (z^2, z^3 w^5 + z w^3 + w^2)$.
Then $\Gamma_f = \{ ( \mu ,\nu ) : \mu = \nu^{-1} \in S^1 \}$.
Moreover, $f$ is semiconjugate to $f_0 (z,w) = (z^2, w^5 + w^3 + w^2)$
by $\pi (z, w) = (z, z^{-1} w):$ 
$\pi f_0 = f \pi$.
\end{example}

Although the map $f$ below does not satisfy any condition 
in Corollary {\rmfamily \ref{Gamma = F}},
it again follows that $\Gamma_f = \mathcal{E}_f$.

\begin{example}\label{ex5} 
Let $f(z,w) = (z^2, (z^l - 1) w^2)$, where $l \geq 1$.
It follows from Proposition {\rmfamily \ref{inverse}}
that $\Gamma_f \supset \{ \mu^l = 1 \} \times S^1$.
Let $\sigma (z) = \mu z$ be the first component of $\gamma$ in $\Gamma_f$. 
Lemma {\rmfamily \ref{sym of b_d}} 
implies that $|\mu^l z^l - 1| = |z^l - 1|$
on a dense subset of $J_p = S^1$.
Hence $\mu^l = 1$ 
and so $\Gamma_f = \{ \mu^l = 1 \} \times S^1$.
\end{example}

In particular, 
the groups of symmetries of Examples 
from {\rmfamily \ref{ex2}} to {\rmfamily \ref{ex5}}
are infinite.

\section{Infinite symmetries}

In this section,
we classify the polynomial skew products
whose Julia sets have infinitely many symmetries.
We first show that these maps in normal form
are classified into four types
in Section 5.1.
We then remove the assumption of normality
and show that the normalized rational skew products of these maps
are also classified into four types
in Section 5.2.

Toward the classification,
we prepare a lemma.

\begin{lemma}\label{semi-conj} 
Let $q(z,w)$ be a polynomial.
If there exist nonzero integers $r$ and $s$ and positive integer $\delta$
such that $q(z^r, z^s w) = z^{s \delta} q(1,w)$,
then $f(z,w) = (z^{\delta}, q(z,w))$ is semiconjugate to 
$f_0(z,w) = (z^{\delta}, q(1,w))$ by $\pi (z,w)=(z^r, z^s w):$ 
$\pi f_0 = f \pi$.
Moreover, 
$q$ consists of the terms of the form $c_j z^{n_j} w^j$,
where $c_j$ is a constant and $n_j = (\delta - j)s/r$
for $0 \leq j \leq \deg q$.
\end{lemma}

The former statement of
this lemma holds even if $q$ is a rational function$;$
we apply this lemma 
for the normalized rational skew products in Section 5.2.
The equation $n_j = (\delta - j)s/r$ follows from
the identity $r n_j + s j = s \delta$.

\subsection{Classification of the maps in normal form}

We first assume that polynomial skew products are in normal form. 

\begin{theorem}\label{infinite, normal} 
Let $f$ be in normal form.
Then $\Gamma_f$ is infinite if and only if one of the following holds$:$
\begin{enumerate} [(i)]
\item $f(z,w)=(z^{\delta}, z^l w^d)$,
\item $f(z,w)=(z^{\delta}, q(w))$,
\item $f(z,w)=(p(z), b_d (z) w^d)$, or
\item $f(z,w)=(z^{\delta}, q(z,w))$
      and it is semiconjugate to $(z^{\delta}, q(1,w))$ 
      by $\pi (z,w)$ $=$ $(z^r, z^s w)$
      for some nonzero coprime integers $r$ and $s$.
      If $l = 0$, then $\delta = d$ and $s / r > 0$.
      If $l \neq 0$, then $\delta \neq d$ and
      $s / r = l / (\delta - d)$.
\end{enumerate}
To avoid overlap,
we assume that $q(w) \neq w^d$ in (ii),
$p(z) \neq z^{\delta}$ or $b_d(z) \neq z^l$ in (iii),
and $q(z,w) \neq b_d (z) w^d$ in (iv).
Each condition is equivalent to the following$:$
\begin{enumerate} [(i)']
\item $J_f = S^1 \times S^1$,
\item $J_f = S^1 \times J$,
\item $J_f = \bigcup_{z \in J_p} \{ z \} \times C_z$,
\item $J_f = \bigcup_{z \in S^1} \{ z \} \times z^{s/r} J$,
\end{enumerate}
for a one-dimensional Julia set $J$ which is not $S^1$
and for a circle $C_z$ about the origin which can be empty.
More precisely, $J$ in (ii)' is $J_q;$ 
a circle $C_z$ in (iii)' is $\{ w : \log |w| = - \Phi (z) \};$
and $J$ in (iv)' is $J_{q(1,w)}$.
To avoid overlap, we assume in (iii)' that 
$J_p$ is not $S^1$ or $C_z$ is not all the same over $J_p$.
Note that $J_f$ in (i)', (ii)' or (iv)' is compact. 
Or equivalently,
\begin{enumerate} [(i)'']
\item $\Gamma_f = S^1 \times S^1$,
\item $\Gamma_f = S^1 \times \Sigma$,
\item $\Gamma_f = \Sigma \times S^1$,
\item $\Gamma_f = \{ (\mu, \nu) \in S^1 \times S^1 : \mu^{as} = \nu^{ar} \}$,
\end{enumerate}
for a finite group $\Sigma$ in $S^1$
and for an integer $a \neq 0$.
In particular, $\Gamma_f$ is compact.
More precisely, $\Sigma$ in (ii)'' is $\Sigma_q$,
and $\Sigma$ in (iii)'' is a subgroup of $\Sigma_p$ 
which includes $\Sigma_p \cap \{ \sigma (z) = \mu z : b_d(\mu z) = \mu^l b_d(z) \}$. 
\end{theorem}

\begin{proof} 
The main part of the proof is to show that 
if $\Gamma_f$ is infinite,
then $f$ is one of the forms from $(i)$ to $(iv)$.
For the equivalence of the conditions
in terms of the form of $f$ and
of the shapes of $J_f$ and $\Gamma_f$,
it is enough to show that $(iv)$ implies $(iv)$'',
which is mentioned at the end of this proof. 

Assume that $\Gamma = \Gamma_f$ is infinite.
Let $\pi_z$ and $\pi_w$ be the projections
to the $z$ and $w$ coordinates, respectively.
Then we have three cases:
$\pi_z (\Gamma)$ is finite, $\pi_w (\Gamma)$ is finite,
or $\pi_z (\Gamma)$ and $\pi_w (\Gamma)$ are both infinite.

We first show that $\pi_z (\Gamma)$ being finite induces $(iii)$.
If $\pi_z (\Gamma)$ is finite, 
then $\Gamma$ contains $\{ 1 \} \times S^1$
since it is an infinite subgroup of $S^1 \times S^1$
and since $\partial K_z$ is closed  for any $z$ in $J_p$.
Thus $\partial K_z$ is a circle about the origin for any $z$ in $J_p^*$. 
Because $q_z$ maps this circle to another circle,
one can show that $q_z(w) = b_d(z) w^d$ on $J_p^* \times \mathbb{C}$
by using the fiberwise B\"{o}ttcher function $\varphi_z$,
which implies $(iii)$.

Assume that $\pi_z (\Gamma)$ is infinite from now on.
Then $p(z) = z^{\delta}$ and $J_p = S^1$. 
Moreover, we show that $b_d (z) = z^l$.
It follows from Lemma {\rmfamily \ref{sym of b_d}} 
that $|b_d (\mu z)| = |b_d (z)|$ 
for any $z$ in $J_p^* \cap \sigma^{-1} J_p^*$ and 
for any $\mu$ in a dense subset $\pi_z (\Gamma)$ of $S^1$.
Thus $b_d$ maps $S^1$ to a circle about the origin.
Therefore, $b_d (z) = z^l$. 
In particular, $J_f$ and $\Gamma_f$ are compact.

We next show that $\pi_w (\Gamma)$ being finite induces $(ii)$.  
If $\pi_w (\Gamma)$ is finite, 
then $\Gamma$ contains $S^1 \times \{ 1 \}$
since it is compact.
Thus $\partial K_z$ is independent of $z$ in $J_p$;
it coincides with the Julia set $J$ of a polynomial.
Hence one can show that
the polynomials $q_z$ on fibers over $J_p$ 
differ only in terms of the symmetries of $J$
by using $\varphi_z$.
Because the number of the symmetries of $J$ is finite
and because $J_p$ is connected,
the polynomials are all the same,
which implies $(ii)$.

Now let $\pi_w (\Gamma)$ also be infinite. 
If $q(z,w) = z^l w^d$, then we get $(i)$. 
Let $q \neq z^l w^d$ and 
$z^{n_j} w^{m_j}$ be a term of $q$ 
with a nonzero coefficient for $m_j \leq d$.
Then it follows from Corollary {\rmfamily \ref{cor of fun-eq, b = z^l}}
that $\mu^{n_j} \nu^{m_j} = \mu^{l} \nu^{d}$ 
for any symmetry $(\mu, \nu)$ in $\Gamma$.
The equation $\mu^{n_j - l} = \nu^{d - m_j}$ implies 
that if $m_j = d$ then $n_j = l$
since we may assume that $\mu^n \neq 1$ for any integer $n \neq 0$.
Let $m_j < d$ from now on. 
The equations $\mu^{n_j - l} = \nu^{d - m_j}$ 
and $\mu^{n_i - l} = \nu^{d - m_i}$ imply that 
$\mu^{(n_j - l)(d - m_i)} = \mu^{(n_i - l)(d - m_j)}$,
which implies that $(n_j - l)(d - m_i) = (n_i - l)(d - m_j)$.
Therefore, the ratio of $n_j - l$ and $d - m_j$ does not depend on $j$.
Let $s$ and $r$ be nonzero coprime integers 
whose ratio is equal to that of $n_j - l$ and $d - m_j$;
\[
\dfrac{n_j - l}{d - m_j} = \dfrac{n_i - l}{d - m_i} =: \dfrac{s}{r}.
\]
Then $q(z^r, z^s w) = z^{rl + sd} q(1,w)$.

Moreover, 
we can obtain the identity $rl + sd = s \delta$ as follows.
Let $a_j$ be an integer such that 
$n_j - l = a_j s$ and $d - m_j = a_j r$,
and let $a$ be the greatest common divisor of $a_j$.
By Corollary {\rmfamily \ref{Gamma = F}},
$(1, \mu^{l} \nu^{d - \delta}) 
= (\mu^{\delta}, \mu^{l} \nu^{d}) - (\mu^{\delta},\nu^{\delta})$ 
belongs to $\Gamma$ for any $(\mu, \nu)$ in $\Gamma$.
Since $\Gamma \subset \{ \mu^{as} = \nu^{ar}  \}$, we have 
$(\mu^{l} \nu^{d - \delta})^{ar} =1$.
Therefore, 
$\mu^{a \{ rl + s(d - \delta) \}} = 1$  
and so $rl + s(d - \delta) = 0$.
Consequently, 
$q(z^r, z^s w) = z^{s \delta} q(1,w)$,
which implies $(iv)$ as in Lemma {\rmfamily \ref{semi-conj}}.
More precisely, 
it follows from Proposition {\rmfamily \ref{poly, order}}
that the order of $\Sigma_{q(1,w)}$ is equal to the absolute value of $ar$.
Hence $\Gamma = \{ \mu^{as} = \nu^{ar}  \}$.
\end{proof}

Favre and Guedj showed in \cite[Proposition 6.5]{fg} that, 
for a map $f$ of the form $(iii)$,
if $b^{-1}(0) \cap J_p \neq \emptyset$
then $G_z$ restricted to $J_p \times \mathbb{C}$
is everywhere discontinuous,
and the closure of $J_f$ coincides with $J_p \times \mathbb{C}$.
Hence the groups of the symmetries of $J_f$ and of its closure
may be different. 

\subsection{Classification of normalized rational skew products}

We saw in Section 4.2 that 
the birational map $h$ conjugates $f$ to
the normalized rational skew product $\tilde{f}$: 
$h f = \tilde{f} h$.
Note that $h$ also conjugates a symmetry $\gamma$,
which corresponds to $\mu$ and $\nu$,
to $\tilde{\gamma} (z,w) = (\mu z, \nu w)$.
Let $\tilde{f} (z,w) = (\tilde{p} (z), \tilde{q} (z,w))$ and 
$\tilde{q} (z,w) = \tilde{b}_d (z) w^d 
+ \tilde{b}_{d - 1} (z) w^{d - 1} + \cdots + \tilde{b}_0 (z)$.
Then $\tilde{p}$ and $\tilde{b}_d$ are polynomial 
and $\tilde{b}_{d - 1} \equiv 0$.

\begin{theorem}\label{infinite, normalized}
Let $\Gamma_f$ be infinite. 
Then the rational map $\tilde{f}$ is one of the following$:$
\begin{enumerate} [(i)]
\item $\tilde{f} (z, w) = (z^{\delta}, z^l w^d)$,
\item $\tilde{f} (z, w) = (z^{\delta}, \tilde{q} (w))$,
\item $\tilde{f} (z, w) = (\tilde{p} (z), \tilde{b}_d (z) w^d)$, 
\item $\tilde{f} (z, w) = (z^{\delta}, \tilde{q} (z,w))$
      and it is semiconjugate to $(z^{\delta}, \tilde{q} (1,w))$ 
      by $\pi (z,w)$ $=$ $(z^r, z^s w)$
      for some nonzero coprime integers $r$ and $s$.
      If $l = 0$, then $\delta = d$ and $s / r > 0$.
      If $l \neq 0$, then $\delta \neq d$ and
      $s / r = l / (\delta - d)$.
\end{enumerate}
Except the case $s/r < 0$ in $(iv)$, 
the map $\tilde{f}$ is polynomial. 
To avoid overlap,
we assume that $\tilde{q} (w) \neq w^d$ in (ii),
$\tilde{p} (z) \neq z^{\delta}$ or $\tilde{b}_d(z) \neq z^l$ in (iii),
and $\tilde{q} (z,w) \neq \tilde{b}_d (z) w^d$ in (iv).
\end{theorem}

\begin{proof} 
The proof of the former statement is similar to that of 
Theorem {\rmfamily \ref{infinite, normal}};
we only give the outline.
Let $\Gamma = \Gamma_f$ be infinite.
If $\pi_z (\Gamma)$ is finite,
then $\Gamma \supset \{ 1 \} \times S^1$
and so $\tilde{q} (z,w) = \tilde{b}_d (z) w^d$.
If $\pi_z (\Gamma)$ is infinite,
then $\tilde{p} (z) = z^{\delta}$ and $\tilde{b}_d (z) = z^l$.
In particular, $J_f$ and $\Gamma_f$ are compact
and $h$ is homeomorphic on $J_p \times \mathbb{C}$
since $b_d^{-1} (0) \cap J_p = \emptyset$.
If $\pi_w (\Gamma)$ is finite,
then $\Gamma \supset S^1 \times \{ 1 \}$
and so $\tilde{q} (z,w) = \tilde{q} (w)$.
We now show that 
$\pi_z (\Gamma)$ and $\pi_w (\Gamma)$ being infinite implies $(iv)$, 
assuming $\tilde{q} \neq z^l w^d$.
In this case, 
$\tilde{q} (z,w) = z^l w^d + \sum c_j z^{n_j} w^{m_j}$,
where $c_j$ is a nonzero coefficient,  
$n_j$ can be negative, and $0 \leq m_j \leq d$. 
Because Corollary {\rmfamily \ref{cor of fun-eq, b = z^l}}
still holds for $\tilde{f}$,
we find that $\tilde{q} (\mu z, \nu w) = \mu^l \nu^d \tilde{q} (z,w)$ 
for any $(\mu, \nu)$ in $\Gamma_f$.
The same argument as in the proof of 
Theorem {\rmfamily \ref{infinite, normal}}
implies that 
$\tilde{q} (z^r, z^s w) = z^{s \delta} \tilde{q} (1,w)$,
which implies $(iv)$ as in Lemma {\rmfamily \ref{semi-conj}}.

Let us show the later statement. 
If $\tilde{f}$ is the form of $(i)$ or $(iii)$, then it is clearly polynomial.
In the case $(ii)$, 
it follows from the commutative diagram $h f = \tilde{f} h$ that 
$\zeta_z$ is a polynomial. 
Hence $h$ is polynomial and so is $\tilde{f}$. 
Let $\tilde{f}$ be the form of $(iv)$.
Then the identity $r n_j + s m_j = s \delta$
implies that $n_j = (\delta - m_j)s/r$.
If $s/r > 0$, then $\delta \geq d$ and so $n_j > 0$.
\end{proof}

As Theorem {\rmfamily \ref{infinite, normal}} states,
these conditions in terms of the form of $\tilde{f}$
can be replaced with the shape of $J_{\tilde{f}}$ or $\Gamma_f$.


\begin{thebibliography}{9}
\bibitem{b}
A. F. Beardon, 
{\it Symmetries of Julia sets}, 
Bull. London Math. Soc. \textbf{22} (1990), 576-582.
\bibitem{fg}
C. Favre \and V. Guedj,
{\it Dynamique des applications rationnelles des espaces multiprojectifs},
Indiana Univ. Math. J. \textbf{50} (2001), 881-934.
\bibitem{u}
K. Ueno,
{\it Symmetries of Julia sets of nondegenerate polynomial skew products on $\mathbb{C}^{2}$},
Michigan Math. J. \textbf{59} (2010), 153-168.
\end{thebibliography}
\end{document}